\documentclass[12pt]{article}
\usepackage{amsmath,amssymb,amsthm}
\usepackage{amsfonts}

\usepackage{amsfonts}
\newcommand \R      {\mathbb R}
\newcommand \be     {\begin{equation}}
\newcommand \ee     {\end{equation}}

\newtheorem{theo}{Theorem}[section]
\newtheorem{lemm}[theo]{Lemma}

\newtheorem{prop}[theo]{Proposition}
\theoremstyle{definition}\newtheorem{defi}[theo]{Definition
}

\def\B{\mathcal{B}}

\def\H{\mathbb{H}}

\def\R{\mathbb{R}}

\def\bih{\operatorname{Bih}}
\def\har{\operatorname{Har}}
\def\inf{\operatorname{Infr}}
\def\pol{\operatorname{Pol}}
\def\partialbar{\overline{\partial}}

\def\vc#1{\hspace{-.1ex}\vec{\hspace{.1ex}{#1}}}

\newcommand{\ZerR}[3]{\hspace{.1ex}{}^{\R#3}\hspace{-.5ex}\mbox{Zer}_{#1}^{m}}
\newcommand{\ZerQ}[3]{\hspace{.1ex}{}^{\H#3}\hspace{-.5ex}\mbox{Zer}_{#1}^{m}}

\def\Y{\underline{Y}}
\def\Z{\underline{Z}}

\begin{document}

\begin{center}\Large
 Reduced-quaternion inframonogenic\\ functions   on   the ball 
\end{center}

\vspace{2ex}
\begin{center}
 C.~Álvarez$^{a}$, J.~Morais$^{b}$ and R.~Michael Porter$^{c}$\\[2ex]
\parbox{.8\textwidth}{$^{a}$Departamento de Ingenierías, Universidad Iberoamericana León, Blvd. Jorge Vértiz Campero 1640, Col.~Cañada de Alfaro, León, Guanajuato, Apdo.~Postal 1-26  C.P. 37238 Mexico \\[.5ex] $^{b}$Department of Mathematics, ITAM, R\'io Hondo \#1, Col.~Progreso Tizap\'an, Mexico City, C.P. 01080 Mexico \\[.5ex] $^{c}$Department of Mathematics, CINVESTAV-Quer\'etaro, Libramiento Norponiente \#2000, Fracc.~Real de Juriquilla. Santiago de Quer\'etaro, C.P. 76230 Mexico}
\end{center}

\vspace{2ex}
\begin{center}
  \parbox{.9\textwidth}{Abstract. A function $f$ from a domain in
  $\R^3$ to the quaternions is said to be inframonogenic if
  $\partialbar f\partialbar =0$, where
  $\partialbar = \partial/\partial x_0+ (\partial/\partial
  x_1)e_1+(\partial/\partial x_2) e_2$. All inframonogenic functions
  are biharmonic. In the context of functions $f=f_0+f_1e_1+f_2e_2$
  taking values in the reduced quaternions, we show that the
  homogeneous polynomials of degree $n$ form a subspace of dimension
  $6n+3$. We use them to construct an explicit, computable orthogonal
  basis for the Hilbert space of square-integrable inframonogenic
  functions defined in the ball in $\R^3$.}
\end{center}

\vspace{2ex}
{\it MSC Classification Numbers:} Primary 30G35; Secondary 31B30, 33D45, 35G05, 42C30.
 

\section{Introduction}

In diverse contexts, a solution of a Clifford-algebra equation of the
form $\partialbar f =0$ or $f \partialbar=0$ is known as a ``monogenic
function''.  In \cite{MalonekPenaSommen2010,MalonekPenaSommen2011},
the term ``inframonogenic function'' was used for solutions of a
two-sided or ``sandwich'' second order differential equation of the form
\[ \partialbar f \partialbar =0 .\] In the papers cited as well as in
more recent work
\cite{Dinh2021,MorenoAbreuBory2017,SantiestebanBlayaArciga2021}, the
symbol $\partialbar$ refers to the Dirac operator
$\sum_{i=1}^n\partial_ie_i$, where $\partial_i$ denotes
$\partial/\partial x_i$ while $e_i$ are are the units of the Clifford
algebra under consideration. A rather different theory results when
one uses for $\partialbar$ what is sometimes called the generalized
Cauchy-Riemann (or Fueter) operator
$\partial_0 + \sum_{i=1}^n\partial_ie_i$.

Some natural similarities can be found in the theories independently
of whether or not $\partial_0$ is included in the definition of
$\partialbar$: in any event the set of inframonogenic functions
includes all functions which are both left- and right-monogenic and is
included in the space of biharmonic functions. In either case, a
scalar-valued ``inframonogenic'' function takes a simple form. There
are also some important differences; in particular, the formulation
given in \cite{MalonekPenaSommen2010,MalonekPenaSommen2011} does not
admit separate notions of inframonogenic and of ``antiinframonogenic''
(i.e., solutions of the companion second order differential equation
$\partial f \partial =0$, where $\partial$ denotes the conjugate
Cauchy-Riemann operator). Further, the Cauchy-Riemann version of the operator $\partialbar (\cdot) \partialbar$ preserves paravectors, and the 
relation of $\partial_0$ to $\partial$ for monogenic functions gives this theory a special flavor.

To initiate the study of Cauchy-Riemann inframonogenics, we will work
here with reduced quaternions, that is, functions of three real
variables taking values in the subspace of real codimension 1 generated by the
units $e_0$, $e_1$, and $e_2$ in the
even Clifford algebra $\mathcal{C}\ell^{+}_{0,3}$, identified with
the real quaternions $\H$. The operator under consideration is
\[ \partialbar = \partial_0 + \partial_1e_1 + \partial_1e_2. \] Our
approach will be different from the work which has been done in
the context of the Dirac operator. In the interests of computability, our focus will be on the real vector
subspaces of inframonogenic homogeneous polynomials of given degree
$n$, which we show to be of dimension $6n + 3$. Then we construct an
orthogonal basis for square-integrable inframonogenic functions in the
ball in $\R^3$. We give the norms of the basis elements, which permits the computation of
the associated Fourier-type coefficients. It may be noted that since the collection of inframonogenic
functions is not preserved under the operation of multiplication by a
quaternionic constant either on the left or on the right, there is no apparent
way to construct directly from this a basis for inframonogenic functions taking
values in full quaternions.

\section{Inframonogenic functions}

Let
$\H=\{x = x_0e_0 + x_1e_1 + x_2e_2+x_3e_3: x_i \in \R, \, i=0,1,2,3\}$
denote the collection of real quaternions, where $e_0=1$ and
$e_1, e_2, e_3$ are the quaternionic imaginary units satisfying the
multiplication rules $e_1^2 = e_2^2 = e_3^2 = e_1 e_2 e_3 = -1$. One
usually writes $\overline{x}$, $|x|$ for the conjugate and absolute
value operations on $\H$ as in
\cite{GurlebeckSprossig1989,GurlebeckSprossig1997,Kravchenko2003,MoraisGeorgievSproessig2014}. We
identify the Euclidean space $\R^3=\{x = (x_0,x_1,x_2)\}$ with the
real vector subspace of reduced quaternions
 \[   \{x =   x_0e_0 + x_1e_1 + x_2e_2\}\subseteq\H,\]
i.e.\ with vanishing $e_3$ term.
We use the common notation
\begin{align*}
 \partial_i &= \partial/\partial x_i, \quad
\vc\partial = \partial_1e_1 + \partial_2e_2, \quad
 \partial = \partial_0  - \vc\partial, \quad
 \partialbar = \partial_0  + \vc\partial, 
\end{align*}
for first order differential operators on functions
$f\colon\Omega\to\R^3$ on any domain $\Omega\subseteq\R^3$, with
$f=\sum_0^2f_ie_i=f_0+\vc f$ where $f_i\colon\Omega\to\R$.  Some
properties of these operators are given in 
\ref{sec:formulas}. As in
\cite{GuerlebeckHabethaSproessig2008,GuerlebeckHabethaSproessig2016},
  $f$ is called \textit{monogenic} when $\partialbar f= 0$. For
$\R^3$-valued functions this is the same as $f\partialbar= 0$ since
$-e_3 (\partialbar f)e_3 = \overline{f\partialbar}$
(cf. \cite{MoraisHabilitation2021}). Similarly, $f$ is
\textit{antimonogenic} when $\partial f= 0$. Monogenic and
antimonogenic functions are harmonic.

\begin{defi} We say that  
$f$ is \textit{inframonogenic} if $\partialbar f\partialbar = 0$ while 
$f$ is \textit{antiinframonogenic} if $\partial f\partial = 0$.
\end{defi}

When working with symbolic operators, one should take care since
$\partialbar(fg)\partialbar$ is not the same as
$(\partialbar f)(g\partialbar)$ despite the fact that multiplication
in $\H$ is associative; in fact, they are second and first order operators, respectively. Similarly, $f\partialbar g$ has no
clear meaning without parentheses.
Clearly, $f$ is inframonogenic if and only if $\overline{f}$ is
antiinframonogenic. If $f$ is monogenic or antimonogenic, then $f$ is
also inframonogenic.  We have the Laplace operator  
\[ \Delta_3 = \partial_0^2 + \partial_1^2 +\partial_2^2; \] one calls
$f$ \textit{biharmonic} when $\Delta_3 f$ is harmonic, i.e.\
$(\Delta_3)^2f=0$. Inframonogenic functions bear a similar relation to
biharmonic functions as monogenic functions bear to harmonic
functions.

\begin{prop}\label{prop:partialbpartial}
  If \/ $b$ is biharmonic and real-valued, then $\partial b\partial$
  is inframonogenic and $\partialbar b\partialbar$ is
  antiinframonogenic. If $f$ is inframonogenic or antiinframonogenic,
  then every component $f_i$ of $f$ is biharmonic.
 \end{prop}

In particular, every inframonogenic function is real-analytic \cite{ACL1983}.
A fundamental property of the two-sided operator $\partial(\cdot)\partial$ is that it respects the absence of the $e_3$ term: 

\begin{prop}\label{prop:R3conserved} If $f\colon\Omega\to\R^3$, then
  $\partial f\partial,\,\partialbar f\partialbar\colon\Omega\to\R^3$ as well.
\end{prop}

\begin{proof}
  By Proposition \ref{prop:bilateral},   $\vc\partial f_0\vc\partial$ and
  $\vc\partial \vc f\vc\partial$ take values  in $\R^3$, and obviously so do
  $\partial_0 f_0\vc\partial$ and $\vc\partial \vc f\vc\partial$.
\end{proof}

 The following is analogous to the often-used fact that $\partial f=2f_0$ for monogenic $f$.
\begin{prop} If $f$ is monogenic, then $\partial f+f\partial=4\partial_0f$.
\end{prop}
\begin{proof} First observe that
  \begin{align*}
    2\partialbar f &= 2(\partial_0f_0+\partial_0\vc f
               + \vc\partial f_0   + \vc\partial\vc f) = 0, \\
    \partial f &= \partial_0f_0+\partial_0\vc f
               - \vc\partial f_0  - \vc\partial\vc f , \\    
    f  \partial  &= \partial_0f_0+\partial_0\vc f
               - \vc\partial f_0  - \vc f\vc\partial  .
  \end{align*}
  Then add.  
\end{proof}

The following fact shows that scalar-valued inframonogenic
functions in and of themselves are not particularly interesting.
\begin{prop}
  Let $f_0\colon\Omega\to\R$ be an inframonogenic scalar-valued
  function. Then (locally) one has
\[ f(x_0,x_1,x_2) = c_0(2x_0^2+x_1^2+x_2^2)+c_1x_0 + c_2+h(x_1,x_2),
\]
where $c_i\in\R$ are constants and $h$ is a harmonic function.
\end{prop}

\begin{proof}
$\partialbar f_0\partialbar=0$ says $\partialbar^2f_0=0$ that
is equivalent to
\[ (\partial_0^2-\partial_1^2-\partial_2^2)f_0=0,
  \quad  \partial_0\vc\partial f_0=0.
\]
The second equation is the system
$\partial_0\partial_1f_0=\partial_0\partial_2f_0=0$, which says that
$\partial_0f_0$ is independent of $x_1,x_2$ while $\partial_1f_0$,
$\partial_2f_0$ are independent of $x_0$:
\[ \partial_0f_0 = a_0(x_0), \quad \partial_1f_0=a_1(x_1,x_2),\quad 
  \partial_2f_0=a_2(x_1,x_2).
\] 
From the first of these equations
\[ f_0(x_0,x_1,x_2) =   a(x_0) +  b(x_1,x_2),
\]
for some $b(x_1,x_2)$ where $a'(x_0)=a_0(x_0)$, and this general
expression for $f_0$ satisfies the latter two. Now
\[ 0= (\partial_0^2-\partial_1^2-\partial_2^2)f_0
  = a''(x_0) - \Delta_2b(x_1,x_2).
\]
This implies $ a''(x_0) = \Delta_2b(x_1,x_2)= 4c_0$ for some
constant $c_0$. From this,
\[ a(x_0) = 2c_0x_0^2 + c_1x_0 + c_2 \]
and
\[ b(x_1,x_2) = c_0(x_1^2+x_2^2) + h(x_1,x_2)
\]
for some harmonic $h$.
\end{proof}

\section{Homogeneous polynomials}

Here we collect the facts we will need about homogeneous polynomials
of three variables. Let $\pol_n$ denote the collection of homogeneous
polynomials of degree $n$ in the variables $x_0,x_1,x_2$.  Let
$\har_n\subseteq\bih_n$ be the subspaces of real-valued harmonic and
biharmonic polynomials.  It is well known  \cite{Mus1968,Sansone1959} that
\begin{align}
  \dim \har_n &= 2n+1, \quad n\ge0, \label{eq:dimharn} \\
  \dim \bih_n &= 4n-2, \quad n\ge2 \label{eq:dimbihn}
\end{align}
while  $\dim \bih_0=1$,  $\dim \bih_1=3$ since all polynomials of
degree 0 or 1 are biharmonic. 

We need the following technical lemma.
\begin{lemm} \label{lemm:x2h} If $h$ (in any domain) is harmonic, then $b(x)=|x|^2h(x)$
is biharmonic.
\end{lemm}

\begin{proof} Since
\[ \Delta_3 b = \Delta_3(|x|^2)h+\nabla|x|^2\cdot h + |x|^2\Delta_3 h=
  6h+2x\cdot\nabla h = 6h+2\sum x_i\,\partial_ih,\]
we have
\begin{align*}
\Delta_3\Delta_3 b &= 2\sum \Delta_3(x_i\partial_ih)=
  2\sum(\Delta_3 x_i + \nabla x_i\cdot \nabla\partial_ih + x_i\Delta_3\partial_i h)\\
  &= 2 \sum \partial_i^2h = 0. \qedhere
\end{align*}
\end{proof}

We will use the well-known $2n+1$ solid spherical harmonics
\cite{Andrews1998,Sansone1959} in spherical coordinates
$x_0=\rho\cos\theta$, $x_1=\rho\sin\theta\cos\varphi$,
$x_2=\rho\sin\theta \sin\varphi$:
\begin{align} \label{eq:esphecoor}
  U_{n,m}^+(x) &= \rho^n P_n^m(\cos\theta) \cos (m\varphi),\quad 0\le m\le n,\nonumber\\
  U_{n,m}^-(x) &= \rho^n P_n^m(\cos\theta) \sin (m\varphi),\quad 1\le m\le n,
\end{align}
where $P_n^m$ is the ``associated Legendre polynomial'' of degree $n$
and order $m$ \cite{Hobson1931}.  Recall that $P_n^m(t)$ is a polynomial or
$\sqrt{1-t^2}$ times a polynomial. However, all $U^\pm_{n,m}$ are in
$\pol_n$ and indeed form a basis of $\har_n$. Thus, the $2(n-2)+1=2n-3$
``proper'' (i.e.\ not harmonic) solid spherical biharmonics
\begin{align}
  b_{n,m}^\pm  = |x|^2 U_{n-2,m}^\pm  
\end{align}
which we obtain from Lemma \ref{lemm:x2h} are linearly
independent. The total number of elements of
\begin{align} \label{eq;bihbasis}
  \{U_{n,m}^\pm\}\cup\{b_{n,m}^\pm\} 
\end{align}
is $4n-2$, and they are linearly independent because the factors $\cos (m\varphi)$,
$\sin (m\varphi)$, $\cos (m'\varphi)$, $\sin (m'\varphi)$ are linearly independent
for $m'\not=m$, while when $m'=m$, the factors $P_n^m(\cos\theta)$,
$P_{n-2}^m(\cos\theta)$ are linearly independent. By
\eqref{eq:dimbihn}, we have that \eqref{eq;bihbasis} is a basis for
$\bih_n$ when $n\ge2$. In fact, for $\Omega=B_1(0)=\{x\colon\ |x|<1\}$
it follows from Lemma \ref{lemma_orthogonality_spherical_harmonics} below that \eqref{eq;bihbasis} is an orthogonal set with respect to the
inner product
\begin{align} \label{scalar-inner-product} 
  \langle f,g \rangle =  \int\!\!\int\!\!\int_{B_1(0)}
  (f_0g_0 + f_1g_1 + f_2g_2 ) \,dV  
\end{align}
in the Hilbert space $L^2(B_1(0))$, and by \eqref{eq:dimbihn} is a
basis for $\bih_n \cap L^2(B_1(0))$.

\begin{lemm} \label{lemma_orthogonality_spherical_harmonics}
For arbitrary $k, k' = 0, 1, \dots$, the following relation holds:
\begin{align} \label{L2norms-solid-harmonics-general}
 \langle \rho^{2k} U_{n,m}^{\pm}, \rho^{2k'} U_{n',m'}^{\pm} \rangle 
= \frac{2(1+\delta_{0,m})\pi}{(2(k+k')+(n+n')+3)(2n+1)} \frac{(n+m)!}{(n-m)!} \, \delta_{n,n'} \delta_{m,m'} 
\end{align}
for $n, n' \geq 0$ with $m=0,\ldots,n$ and $m'=0,\ldots,n'$.
\end{lemm}

The following fact will  be useful. The verification is also a direct calculation.
 \begin{lemm} \label{lemm:contrag1}
For $n\ge1$, $\partial_2 U^+_{n,1} =  \partial_1 U^-_{n,1}$. 
\end{lemm}
 
We now turn to our object of interest, the collection $\inf_n$
of inframonogenic homogeneous polynomials of degree $n$. The following
result is inspired by the argument in \cite{Mus1968}.
 
\begin{theo} \label{prop:diminfr}
  The dimension over $\R$ of  $\inf_n$ is $6n+3$.
\end{theo}

\begin{proof}
  All polynomials of degree $\le1$ are inframonogenic.  For $n=0$
  every $\R^3$-valued polynomial is constant, hence $\dim\inf_0=3$,
  while for $n=1$ each component of the polynomial is linear, giving
  $\dim\inf_1=9$.  Let $n\ge2$, and $j_0+j_1+j_2=n$. By Proposition
  \ref{prop:bilateralderivsimple},
\begin{align*}
  \sum_{j_0j_1j_2} &\partialbar\big(b^{j_0j_1j_2}_0x_1^{j_0}x_1^{j_1} x_2^{j_2} + b^{j_0j_1j_2}_1x_1^{j_0}x_1^{j_1} x_2^{j_2}e_1+b^{j_0j_1j_2}_2x_1^{j_0}x_1^{j_1} x_2^{j_2}e_2\big)\partialbar  \\
  &=  \sum_{j_0j_1j_2}
 (-b^{j_0j_1j_2}_0\partial_2^2+\cdots)x_1^{j_0}x_1^{j_1} x_2^{j_2} +
(-b^{j_0j_1j_2}_1\partial_2^2+\cdots)x_1^{j_0}x_1^{j_1} x_2^{j_2}e_1 
  \\&\qquad\qquad  +  (-b^{j_0j_1j_2}_2\partial_2^2+\cdots)x_1^{j_0}x_1^{j_1} x_2^{j_2}e_2  \\[1ex]
  &=  \sum_{j_0j_1j_2}
 (-j_2(j_2+1) b^{j_0j_1j_2}_0 x_1^{j_0}x_1^{j_1} x_2^{j_2-2} +\cdots) \\
 &\qquad\qquad + (-j_2(j_2+1) b^{j_0j_1j_2}_1 x_1^{j_0}x_1^{j_1} x_2^{j_2-2} +\cdots)e_1 \\
 &\qquad\qquad + (-j_2(j_2+1) b^{j_0j_1j_2}_2 x_1^{j_0}x_1^{j_1} x_2^{j_2-2} +\cdots)e_2 \\[1ex]
   &= \sum_{i_0i_1i_2}
  (-(i_2+1)(i_2+2) b^{i_0,i_1,i_2+2}_0  +\cdots)x_1^{i_0}x_1^{i_1} x_2^{i_2}
\\&\qquad\qquad  + (-(i_2+1)(i_2+2) b^{i_0,i_1,i_2+2}_1  +\cdots) x_1^{i_0}x_1^{i_1} x_2^{i_2}e_1
  \\&\qquad\qquad + (-(i_2+1)(i_2+2) b^{i_0,i_1,i_2+2}_2  +\cdots)x_1^{i_0}x_1^{i_1} x_2^{i_2}e_2   \\[1ex]
  &= \sum_{i_0i_1i_2}
    a^{i_0,i_1,i_2}_0  x_1^{i_0}x_1^{i_1} x_2^{i_2}
    +a^{i_0,i_1,i_2}_1  x_1^{i_0}x_1^{i_1} x_2^{i_2}e_1
    +a^{i_0,i_1,i_2}_2  x_1^{i_0}x_1^{i_1} x_2^{i_2}e_2 ,
\end{align*} 
where each coefficient can be expressed in the form
\begin{align} \label{eq:coefeqs}
  a^{i_0i_1i_2}_k = -(i_2+1)(i_2+2) b^{i_0,i_1,i_2+2}_k
  +\cdots,
\end{align}
$k=0,1,2$. This is a system of $n_e$ equations in the $n_v$ variables
$b^{j_0,j_1,j_2}_k$,
\[ n_e = \frac{3(n+1)(n+2)}{2} ,\;\; n_v= \frac{3n(n-1)}{2} ,  
\]
one equation for each $a^{i_0,i_1,i_2}_k$, relating the
coefficients  $b^{j_0,j_1,j_2}_k$.

We order the $b^{j_0,j_1,j_2}_k$ lexicographically by
$(j_0,j_1,j_2,k)$. Then \eqref{eq:coefeqs} displays the first nonzero
term in the equation indexed by $(i_0,i_1,i_2,k)$. Different equations
have a different first variable.  Therefore the rank of the system
\eqref{eq:coefeqs} is equal to $n_e$, and the dimension of the
solution set is $n_v-n_e=6n+3$. 
\end{proof}

Recalling Proposition \ref{prop:partialbpartial}, we now see that not
every inframonogenic function is of the form $\partial b\partial$
since there are only $4n+2$ linearly independent elements of
$\bih_{n+1}$.  This makes the study of inframonogenic functions more
challenging than monogenic functions, where the dimensions of the corresponding sets of polynomials coincide with the dimensions of the solid spherical harmonics of the next higher degree  (cf.\ Proposition  \ref{prop:X}
below).

\section{Bases of inframonogenic polynomials}

A basis for homogeneous monogenic polynomials taking values in the
reduced quaternions of given degree is well-known. For $n\ge0$, let
\begin{align*}
  X^+_{n,m} &=  \partialbar U^+_{n+1,m}  ,\quad 0\le m\le n+1,\\
 X^-_{n,m} &=  \partialbar U^-_{n+1,m}  ,\quad 1\le m\le n+1.
\end{align*}

\begin{prop}[\cite{Cacao2004,CacaoGuerlebeckBock2006,MoraisGurlebeck2012}] \label{prop:X} For each $n\ge0$, the collection
  \[ \{X^+_{n,m}\colon 0\le m\le n+1\} \cup\{X^-_{n,m}\colon
    1\le m\le n+1\} \] is a basis for the real vector space of
  monogenic homogeneous polynomials of degree n, which therefore has
  dimension $2n+3$. The union of these collections for $0\le n<\infty$
  set is orthogonal in the Hilbert space $L^2(B_1(0))$ and is a
  Hilbert basis for the subspace of square-integrable monogenic
  functions. Further, these graded bases satisfy the Appell-type property
 \begin{align} \label{eq:appell}
   \partial X^\pm_{n,m} =  X^\pm_{n,m}\partial = 2 (n + m + 1)  X^\pm_{n-1,m},
   \quad 0\le m\le n,
 \end{align}
while $\partial X^\pm_{n,n+1}=   X^\pm_{n,n+1}\partial =0$.
\end{prop}

We are now in a position to make the following definition.
\begin{defi}\label{defi:basics}   The basic inframonogenic polynomials
  are defined as follows. For degree $n=0$, we have three constant functions:
  \[ \B_0= \{  e_0,\ e_1,\ e_2 \} . \]
For degree $n=1$, we take the linear monomials as follows:
\[ \B_1= \{ x_0 e_0,\ x_1 e_0,\ x_2 e_0,\
      x_0e_1,\ x_1 e_1,\ x_2 e_1,\  x_0 e_2,\ x_1 e_2,\ x_2 e_2 \}. \]
For degrees $n\ge2$, we have three types of homogeneous polynomials:

  \underline{Type 0} ($2n+3$ monogenic polynomials: 
  $n+2$ even, $n+1$ odd):
 \begin{align*}  
     X^+_{n,m}     ,\quad 0\le m\le n+1, \\
     X^-_{n,m}    ,\quad 1\le m\le n+1.
  \end{align*}
  
\underline{Type 1} ($2n-1$ polynomials: $n$ even, $n-1$ odd):
\begin{align*}
 Y^+_{n,m} &=   
    \overline{x}\, X^+_{n-1,m} + X^+_{n-1,m} \,\overline{x} +
     2(n+m) |x|^2  X^+_{n-2,m} , \quad 0\le m\le n-1, \\
  Y^-_{n,m} &=   
    \overline{x}\, X^-_{n-1,m} + X^-_{n-1,m} \,\overline{x} + 
     2(n+m) |x|^2  X^-_{n-2,m} , \quad 1\le m\le n-1.
   \end{align*}
   
   \underline{Type 2} ($2n+1$ polynomials: 1 contragenic (for $m=0$), $n$ even, $n$ odd):
 \begin{align*}
     \Z^+_{n,0} &= \big(
    (e_2\partialbar+\partialbar e_2)U^+_{n+1,1} -    
    (e_1\partialbar+\partialbar e_1)U^-_{n+1,1} \big), \\
 \Z^+_{n,m} &=  
 \overline{x}\, X^+_{n-1,m} + X^+_{n-1,m} \,\overline{x}
     - U^+_{n,m}  + (n+m)  |x|^2 \, X^+_{n-2,m}  ,
   \quad 1\le m\le n,\\
 \Z^-_{n,m} &= 
 \overline{x}\, X^-_{n-1,m} + X^-_{n-1,m} \,\overline{x}
   - U^-_{n,m}  + (n+m)  |x|^2 \, X^-_{n-2,m} ,    \quad 1\le m\le n. 
 \end{align*}
\end{defi}

The notation $\Z$ is intended to reflect that these functions are a provisional construction, to be replaced in \eqref{New_basis} below with the definitive basis. In these definitions, it is tacitly assumed that $X^\pm_{n,m}=0$ when $n<0$. We will see that the collections of $6n+3$ functions given
above lie in $\inf_n$ and are linearly independent, and therefore form
a basis of $\inf_n$. Due to the lengthy calculations, we will go by
steps; along the way, we will express the basic elements in terms of the
solid spherical harmonics \eqref{eq:esphecoor}.

We will need the following fact.
\begin{lemm}\label{lemm:simplecases}  For any monogenic function $f$,
\begin{align*}
  \partialbar(\overline x f+ f\overline x)\partialbar &=
  4\partialbar f_0 ; \\
  \partialbar(|x|^2f)\partialbar &= -2\,\overline{f} .
\end{align*}
\end{lemm}

\begin{proof}
   By Proposition \ref{prop:leibnitz} and the monogenicity of $f$,
\begin{align*}
\partialbar(\overline x f+ f\,\overline x) &=
6f+2e_3\big(x_1(\partial_2f)
-x_2(\partial_1f)\big) + 2e_3\big(f_1e_2-f_2e_1\big),
\end{align*}
so applying   $\partialbar$ on the  right leaves
\begin{align*}
 \partialbar( \overline x f+ f\,\overline x)\partialbar &=
    2e_3\big( (x_1(\partial_2f))\partialbar
        - (x_2(\partial_1f))\partialbar  )
   + (f_1e_2)\partialbar-(f_2e_1)\partialbar   \big).
\end{align*} 
Again by Proposition \ref{prop:leibnitz},
\begin{align*}
\partialbar(\overline x f+ f \overline x)\partialbar &=
 2e_3\big(\, (e_1\partial_2f - 2\partial_2f_2\,e_3)
  - (e_2\partial_1f + 2\partial_1f_1\, e_3 ) \\
  &\quad \quad\  +   ( (f_1\partialbar)e_2 - 2\partial_1f_1\,e_3 ) -
  (f_2\partialbar) e_1-2\partial_2f_2 e_3) \big)\\
&=8(  \partial_1f_1 +\partial_2f_2 ) +2(e_1\partial_1f  +e_2\,\partial_2f)   +2e_3(f_1 \partialbar)e_2-2e_3(f_2\partialbar)e_1.
\end{align*}
Expanding terms into their $e_0, e_1, e_2, e_3 $ parts and grouping,
we have
\[ \partialbar(\overline x f+ f \overline x)\partialbar =
  4(\partial_1f_1+\partial_2f_2) + 2(\partial_1f_0-\partial_0f_1)e_1
  +2(\partial_2f_0-\partial_0f_2)e_2 +4(\partial_1f_2-\partial_2f_1)e_3, \]
which is equal to $4\partialbar f_0$ because $f$ is monogenic. This proves the first assertion.

For the second assertion, since $|x|^2$ is scalar, one can apply the
ordinary Leibnitz rule to obtain
$\partialbar(|x|^2 f) = (\partialbar
|x|^2)f + |x|^2(\partialbar
f)=2xf$ since $\partial_i |x|^2 = 2x_i$. By
Proposition \ref{prop:leibnitz},
  \begin{align*}
  \frac{1}{2}\partialbar(|x|^2 f)\partialbar 
  &= x(f\partialbar) + (x\partialbar)f + 2(f_1e_2-f_2e_1)e_3 )
    = -\overline{f}
  \end{align*}
since $x\partialbar=-1$.
\end{proof}

 \begin{prop}
   $X^\pm_{n,m}, Y^\pm_{n,m}, \Z^\pm_{n,m}\in\inf_n$ for all $n, m$.
\end{prop}
\begin{proof}
  It is clear that by construction, all of these functions are
  homogeneous polynomials of degree $n$ because $U^\pm_{n,m}$ and
  $X^\pm_{n,m}$ have this property. It is easily seen that they are
  $\R^3$-valued, for example, using Lemma \ref{lemm:simplecases}. We
  proceed to verify that they are all indeed inframonogenic.

  The statement for $X^\pm_{n,m}$ with $0 \le m\le n+1$ follows
  immediately from Proposition \ref{prop:X}.

Using Lemma \ref{lemm:simplecases} and the Appell property \eqref{eq:appell}, we see that
\begin{align*}
  \partialbar\, Y^\pm_{n,m}\, \partialbar &=
    4\partialbar[X^\pm_{n-1,m}]_0 + 2(n+m)(-2\overline{X^\pm_{n-1,m}})\\
    &= 2\partialbar(2[X^\pm_{n-1,m}]_0 -\overline  {X^\pm_{n-1,m}})\\
    &=2\partialbar X^\pm_{n-1,m}\\ &=0,
\end{align*}
so the $\Y^\pm_{n,m}$ are inframonogenic.

Next observe that since $U^+_{n+1,n}$ is harmonic,
$(\partial_0^2+\partial_1^2-\partial_2^2)U^+_{n+1,n}=
2\partial_2^2U^+_{n+1,n}$, so
  \begin{align*}
  \partialbar e_2\partialbar U^+_{n+1,n} 
  &= -2(\partial_0\partial_2 +\partial_1\partial_2e_1-\partial_2^2e_2)  U^+_{n+1,n} \\
  &= -2(\partial_0\partial_1 +\partial_1^2e_1-\partial_1\partial_2e_2)  U^-_{n+1,n} \\
  &= U^-_{n+1,n} \partialbar e_1\partialbar  ,
  \end{align*}
  where the second equality follows from Lemma \ref{lemm:contrag1}. Therefore
  \[ \partialbar(e_2\partialbar U^+_{n+1,n} - \partialbar U^-_{n+1,n}e_1)
    \partialbar =0, \]
and similarly $\partialbar( \partialbar e_2U^+_{n+1,n} - _1\partialbar U^-_{n+1,n}e_1 )
    \partialbar =0$. Therefore $\Z^+_{n,0}$ is inframonogenic for all
   $n\ge2$.
   
The verification for $\Z^\pm_{n,m}$ is similar, based on the fact that
\[  \partialbar U^\pm_{n,m} \partialbar =
  \overline{X}^\pm_{n-1,m} \partialbar=2(n+m)\overline{X}^\pm_{n-2,m} .
  \qedhere
 \]  
\end{proof}

Some examples of the basic type 1 and type 2 inframonogenic
polynomials in low degree are exhibited in Tables
\ref{tab:deg2}--\ref{tab:deg14}.

\begin{table}   \centering
\[
 \begin{array}{l|l}   \hline
 Y^+_{2,0} &   8 x_0^2 + 6 x_1^2 + 6 x_2^2 -2 x_0 x_1 e_1  -2 x_0 x_2  e_2  \\[.5ex]
 Y^+_{2,1} &    ( 12 x_0^2 + 12 x_1^2 + 6 x_2^2 )e_1 +  6 x_1 x_2  e_2  \\[.5ex]
 Y^-_{2,1} &    6 x_1 x_2 e_1 + ( 12 x_0^2 + 6 x_1^2 + 12 x_2^2 ) e_2  \\[.5ex] \hline
 \Z^+_{2,0} &    -3 x_0 x_2 e_1 +  3 x_0 x_1  e_2  \\[.5ex]
 \Z^+_{2,1} &   \frac{7}{5} x_0 x_1 + ( \frac{133}{110} x_0^2 + \frac{-21}{110} x_1^2 + \frac{-329}{220} x_2^2 )e_1 + ( \frac{287}{220} x_1 x_2 + \frac{2}{15}   ) e_2  \\[.5ex]
 \Z^+_{2,2} &  -14 x_1^2 + 14 x_2^2   -14 x_0 x_1 e_1 + ( 14 x_0 x_2 + \frac{1}{15} ) e_2  \\[.5ex]
 \Z^-_{2,1} &   \frac{7}{5} x_0 x_2 +  \frac{287}{220} x_1 x_2 e_1 + ( \frac{133}{110} x_0^2 + \frac{-329}{220} x_1^2 + \frac{-21}{110} x_2^2 + \frac{2}{15}   ) e_2  \\[.5ex]
 \Z^-_{2,2} &   -28 x_1 x_2 -14 x_0 x_2 e_1 + ( -14 x_0 x_1 + \frac{1}{15}   ) e_2  \\[.5ex] \hline
 \end{array}
\]
\caption{Basic type 1 and type 2 inframonogenic polynomials of degree $2$.}
\label{tab:deg2}
\end{table}

\def\split{\\&\quad}

\begin{table}[th!]   \centering
\[ 
  \begin{array}{l|l}   \hline
 Y^+_{3,0} &  18 x_0^3 + 15 x_0 x_1^2 + 15 x_0 x_2^2 + ( 6 x_0^2 x_1 + 9 x_1^3 + 9 x_1 x_2^2 )e_1\split  + ( 6 x_0^2 x_2 + 9 x_1^2 x_2 + 9 x_2^3 ) e_2  \\[.5ex]
 Y^+_{3,1} &  -36 x_0^2 x_1 - 33 x_1^3 - 33 x_1 x_2^2 \split + ( 36 x_0^3 + 39 x_0 x_1^2 + 21 x_0 x_2^2 )e_1 + ( 18 x_0 x_1 x_2 ) e_2  \\[.5ex]
 Y^+_{3,2} &  -30 x_0 x_1^2 + 30 x_0 x_2^2 + ( -120 x_0^2 x_1 - 90 x_1^3 - 30 x_1 x_2^2 )e_1 \split + ( 120 x_0^2 x_2 + 30 x_1^2 x_2 + 90 x_2^3 ) e_2  \\[.5ex]
 Y^-_{3,1} &  -36 x_0^2 x_2 - 33 x_1^2 x_2 - 33 x_2^3 + ( 18 x_0 x_1 x_2 )e_1 + ( 36 x_0^3 + 21 x_0 x_1^2 + 39 x_0 x_2^2 ) e_2  \\[.5ex]
    Y^-_{3,2} &  -60 x_0 x_1 x_2 + ( -120 x_0^2 x_2 - 120 x_1^2 x_2 - 60 x_2^3 )e_1 \split + ( -120 x_0^2 x_1 - 60 x_1^3 - 120 x_1 x_2^2 ) e_2  \\[.5ex]
    \hline
  \end{array}
  \]
\caption{ Basic  type 1 inframonogenic polynomials of degree $3$.}
\label{tab:deg13}
\end{table}

\begin{table}[hb!]   \centering
  \[
   \begin{array}{l|l}   \hline
    \Z^+_{3,0} &   ( -6 x_0^2 x_2 + \frac{3}{2} x_1^2 x_2 + \frac{3}{2} x_2^3 )e_1 + ( 6 x_0^2 x_1 + \frac{-3}{2} x_1^3 + \frac{-3}{2} x_1 x_2^2 ) e_2  \\[.5ex]
 \Z^+_{3,1} &  \frac{3942}{1939} x_0^2 x_1 + \frac{-2997}{7756} x_1^3 + \frac{-2997}{7756} x_1 x_2^2 + ( \frac{3537}{1939} x_0^3 + \frac{-4617}{7756} x_0 x_1^2 + \frac{-41607}{7756} x_0 x_2^2 )e_1 \split + ( \frac{18495}{3878} x_0 x_1 x_2 + \frac{3}{28} 
   ) e_2  \\[.5ex]
 \Z^+_{3,2} &  \frac{-18765}{1001} x_0 x_1^2 + \frac{18765}{1001} x_0 x_2^2 + ( \frac{-17145}{1001} x_0^2 x_1 + \frac{1620}{1001} x_1^3 + \frac{2835}{143} x_1 x_2^2 )e_1 \split + ( \frac{17145}{1001} x_0^2 x_2 + \frac{-2835}{143} x_1^2 x_2 + \frac{-1620}{1001} x_2^3 + \frac{1}{14} 
  ) e_2  \\[.5ex]
 \Z^+_{3,3} & \frac{405}{4} x_1^3 + \frac{-1215}{4} x_1 x_2^2 + ( \frac{405}{4} x_0 x_1^2 + \frac{-405}{4} x_0 x_2^2 )e_1 +  ( \frac{-405}{2} x_0 x_1 x_2 + \frac{1}{28} 
   ) e_2  \\[.5ex]
 \Z^-_{3,1} &  \frac{3942}{1939} x_0^2 x_2 + \frac{-2997}{7756} x_1^2 x_2 + \frac{-2997}{7756} x_2^3 + ( \frac{18495}{3878} x_0 x_1 x_2 )e_1\split  + ( \frac{3537}{1939} x_0^3 + \frac{-41607}{7756} x_0 x_1^2 + \frac{-4617}{7756} x_0 x_2^2 + \frac{3}{28} 
   ) e_2  \\[.5ex]
 \Z^-_{3,2} &  \frac{-37530}{1001} x_0 x_1 x_2 + ( \frac{-17145}{1001} x_0^2 x_2 + \frac{-14985}{2002} x_1^2 x_2 + \frac{21465}{2002} x_2^3 )e_1 \split + ( \frac{-17145}{1001} x_0^2 x_1 + \frac{21465}{2002} x_1^3 + \frac{-14985}{2002} x_1 x_2^2 + \frac{1}{14} 
   ) e_2  \\[.5ex]
 \Z^-_{3,3} &  \frac{1215}{4} x_1^2 x_2 + \frac{-405}{4} x_2^3 + ( \frac{405}{2} x_0 x_1 x_2 )e_1 + ( \frac{405}{4} x_0 x_1^2 + \frac{-405}{4} x_0 x_2^2 + \frac{1}{28} 
   ) e_2  \\[.5ex]  \hline
 \end{array}
\]
\caption{ Basic type  2 inframonogenic polynomials of degree $3$.}
\label{tab:deg23}
\end{table}

\begin{table}[!tb]   \centering
\[
 \begin{array}{l|l}   \hline
 Y^+_{4,0} &  32 x_0^4 + 12 x_0^2 x_1^2 - 15 x_1^4 + 12 x_0^2 x_2^2 - 30 x_1^2 x_2^2 - 15 x_2^4 \split + ( 28 x_0^3 x_1 + 33 x_0 x_1^3 + 33 x_0 x_1 x_2^2 )e_1 + ( 28 x_0^3 x_2 + 33 x_0 x_1^2 x_2 + 33 x_0 x_2^3 ) e_2  \\[.5ex]
 Y^+_{4,1} &  -160 x_0^3 x_1 - 150 x_0 x_1^3 - 150 x_0 x_1 x_2^2  \split + ( 80 x_0^4 + 30 x_0^2 x_1^2 - 60 x_1^4 + 30 x_0^2 x_2^2 - 75 x_1^2 x_2^2 - 15 x_2^4 )e_1 \split + ( -45 x_1^3 x_2 - 45 x_1 x_2^3 ) e_2  \\[.5ex]
 Y^+_{4,2} &  180 x_0^2 x_1^2 + 240 x_1^4 - 180 x_0^2 x_2^2 - 240 x_2^4 \split + ( -540 x_0^3 x_1 - 480 x_0 x_1^3 - 180 x_0 x_1 x_2^2 )e_1 + \split ( 540 x_0^3 x_2 + 180 x_0 x_1^2 x_2 + 480 x_0 x_2^3 ) e_2  \\[.5ex]
 Y^+_{4,3} &  420 x_0 x_1^3 - 1260 x_0 x_1 x_2^2 + ( 1260 x_0^2 x_1^2 + 840 x_1^4 - 1260 x_0^2 x_2^2 - 630 x_1^2 x_2^2 - 630 x_2^4 )e_1 \split + ( -2520 x_0^2 x_1 x_2 - 1050 x_1^3 x_2 - 1890 x_1 x_2^3 ) e_2  \\[.5ex]
 Y^-_{4,1} &  -160 x_0^3 x_2 - 150 x_0 x_1^2 x_2 - 150 x_0 x_2^3 + ( -45 x_1^3 x_2 - 45 x_1 x_2^3 )e_1 \split+ ( 80 x_0^4 + 30 x_0^2 x_1^2 - 15 x_1^4 + 30 x_0^2 x_2^2 - 75 x_1^2 x_2^2 - 60 x_2^4 ) e_2  \\[.5ex]
 Y^-_{4,2} &  360 x_0^2 x_1 x_2 + 480 x_1^3 x_2 + 480 x_1 x_2^3 + ( -540 x_0^3 x_2 - 630 x_0 x_1^2 x_2 - 330 x_0 x_2^3 )e_1 \split+ ( -540 x_0^3 x_1 - 330 x_0 x_1^3 - 630 x_0 x_1 x_2^2 ) e_2  \\[.5ex]
   Y^-_{4,3} &  1260 x_0 x_1^2 x_2 - 420 x_0 x_2^3 + ( 2520 x_0^2 x_1 x_2 + 1890 x_1^3 x_2 + 1050 x_1 x_2^3 )e_1 \split+ ( 1260 x_0^2 x_1^2 + 630 x_1^4 - 1260 x_0^2 x_2^2 + 630 x_1^2 x_2^2 - 840 x_2^4 ) e_2  \\[.5ex]\hline
 \end{array}
 \]
\caption{Basic type 1  inframonogenic polynomials of degree $4$.}
\label{tab:deg14}
\end{table}

\begin{table}[!b]   \centering
\[
 \begin{array}{l|l}   \hline  
 \Z^+_{4,0} &  ( -10 x_0^3 x_2 + \frac{15}{2} x_0 x_1^2 x_2 + \frac{15}{2} x_0 x_2^3 )e_1 + ( 10 x_0^3 x_1 + \frac{-15}{2} x_0 x_1^3 + \frac{-15}{2} x_0 x_1 x_2^2 ) e_2  \\[.5ex]
 \Z^+_{4,1} &  \frac{1826}{705} x_0^3 x_1 + \frac{-319}{188} x_0 x_1^3 + \frac{-319}{188} x_0 x_1 x_2^2 \split + ( \frac{1672}{705} x_0^4 + \frac{-1749}{940} x_0^2 x_1^2 + \frac{77}{1410} x_1^4 + \frac{-12089}{940} x_0^2 x_2^2 + \frac{715}{376} x_1^2 x_2^2 + \frac{10417}{5640} x_2^4 )e_1\split + ( 11 x_0^2 x_1 x_2 + \frac{-10109}{5640} x_1^3 x_2 + \frac{-10109}{5640} x_1 x_2^3 + \frac{4}{45} 
   ) e_2  \\[.5ex]
 \Z^+_{4,2} &  \frac{-42438}{947} x_0^2 x_1^2 + \frac{5918}{947} x_1^4 + \frac{42438}{947} x_0^2 x_2^2 + \frac{-5918}{947} x_2^4 \split+ ( \frac{-39358}{947} x_0^3 x_1 + \frac{8998}{947} x_0 x_1^3 + \frac{104940}{947} x_0 x_1 x_2^2 )e_1 \split+ ( \frac{39358}{947} x_0^3 x_2 + \frac{-104940}{947} x_0 x_1^2 x_2 + \frac{-8998}{947} x_0 x_2^3 + \frac{1}{15} 
  ) e_2  \\[.5ex]
 \Z^+_{4,3} &  \frac{12089}{54} x_0 x_1^3 + \frac{-12089}{18} x_0 x_1 x_2^2 \split+ ( \frac{3773}{18} x_0^2 x_1^2 + \frac{-385}{27} x_1^4 + \frac{-3773}{18} x_0^2 x_2^2 + \frac{-7931}{36} x_1^2 x_2^2 + \frac{3157}{36} x_2^4 )e_1 \split+ ( \frac{-3773}{9} x_0^2 x_1 x_2 + \frac{26873}{108} x_1^3 x_2 + \frac{-539}{12} x_1 x_2^3 + \frac{2}{45} 
  ) e_2  \\[.5ex]
 \Z^+_{4,4} &  -924 x_1^4 + 5544 x_1^2 x_2^2 - 924 x_2^4 + ( -924 x_0 x_1^3 + 2772 x_0 x_1 x_2^2 )e_1\split + ( 2772 x_0 x_1^2 x_2 - 924 x_0 x_2^3 + \frac{1}{45} 
   ) e_2  \\[.5ex]
 \Z^-_{4,1} &  \frac{1826}{705} x_0^3 x_2 + \frac{-319}{188} x_0 x_1^2 x_2 + \frac{-319}{188} x_0 x_2^3 \split+ ( 11 x_0^2 x_1 x_2 + \frac{-10109}{5640} x_1^3 x_2 + \frac{-10109}{5640} x_1 x_2^3 )e_1 \split+ ( \frac{1672}{705} x_0^4 + \frac{-12089}{940} x_0^2 x_1^2 + \frac{10417}{5640} x_1^4 + \frac{-1749}{940} x_0^2 x_2^2 + \frac{715}{376} x_1^2 x_2^2 + \frac{77}{1410} x_2^4 + \frac{4}{45} 
   ) e_2  \\[.5ex]
 \Z^-_{4,2} &  \frac{-84876}{947} x_0^2 x_1 x_2 + \frac{11836}{947} x_1^3 x_2 + \frac{11836}{947} x_1 x_2^3 \split + ( \frac{-39358}{947} x_0^3 x_2 + \frac{-38973}{947} x_0 x_1^2 x_2 + \frac{56969}{947} x_0 x_2^3 )e_1 \split+ ( \frac{-39358}{947} x_0^3 x_1 + \frac{56969}{947} x_0 x_1^3 + \frac{-38973}{947} x_0 x_1 x_2^2 + \frac{1}{15} 
  ) e_2  \\[.5ex]
 \Z^-_{4,3} &  \frac{12089}{18} x_0 x_1^2 x_2 + \frac{-12089}{54} x_0 x_2^3 + ( \frac{3773}{9} x_0^2 x_1 x_2 + \frac{539}{12} x_1^3 x_2 + \frac{-26873}{108} x_1 x_2^3 )e_1\split + ( \frac{3773}{18} x_0^2 x_1^2 + \frac{-3157}{36} x_1^4 + \frac{-3773}{18} x_0^2 x_2^2 + \frac{7931}{36} x_1^2 x_2^2 + \frac{385}{27} x_2^4 + \frac{2}{45} 
   ) e_2  \\[.5ex]
 \Z^-_{4,4} &  -3696 x_1^3 x_2 + 3696 x_1 x_2^3 + ( -2772 x_0 x_1^2 x_2 + 924 x_0 x_2^3 )e_1 \split + ( -924 x_0 x_1^3 + 2772 x_0 x_1 x_2^2 + \frac{1}{45} 
  ) e_2  \\[.5ex] \hline
 \end{array}
\]
\caption{Basic  type 2 inframonogenic polynomials of degree $4$.}
\label{tab:deg24}
\end{table}

The next step is to write out the components of
$X^\pm_{n,m}$, $Y^\pm_{n,m}$, $\Z^\pm_{n,m}$ explicitly.  In the
interpretation of the following formulas, when $m=n$, the terms
$U^\pm_{n,m+1}$ must be understood with the interpretation
$U^\pm_{n,n+1}=0$.  For $m=0$ the terms $U^\pm_{n,m-1}$ must be
replaced using the formula
\begin{align}
   U^\pm_{n,-1} = \mp  \frac{ 1}{n(n+1)} \, U^\pm_{n,1} .
\end{align}
This can be proved using properties of the associated Legendre
functions, or simply taken as a definition. Thus, for example, we have
\cite{MoraisGurlebeck2012}
\begin{align*}
 X^\pm_{0,m} & \, =  (m+1) U^\pm_{0,m}  
   +  \frac{1}{2} (m(m+1) U^\pm_{0,m-1} - U^\pm_{0,m+1})e_1 \\[-0.5ex]
 & \qquad \mp \frac{1}{2} (m(m+1) U^\mp_{0,m-1} + U^\mp_{0,m+1})e_2 ,\\[0.5ex]
 X^\pm_{n,n+1} &= (2n+1)(n+1) U^\pm_{n,n} e_1
  \pm (2n+1)(n+1) U^\pm_{n,n} e_2.
\end{align*}
In what follows, for brevity, we will not write out the explicit simplified formulas for $m=0$ and
$m=n$, understanding that the above
substitutions must be carried out when appropriate.

\begin{prop} \label{prop:components}  Let $n\ge2$.
  The inframonogenic polynomials of type 0 can be expressed as
  \begin{align*}
   X^\pm_{n,m} =& \ (n+m+1) U^\pm_{n,m} +\frac{1}{2} \big( (n+m)(n+m+1) U^\pm_{n,m-1} - U^\pm_{n,m+1} \big)e_1 \\   
 &\mp \frac{1}{2} \big( (n+m)(n+m+1)U^\mp_{n,m-1} + U^\mp_{n,m+1} \big)e_2.
\end{align*}
  The inframonogenic polynomials of type 1 can be expressed as
  \begin{align*}
    (2n-1)Y^{\pm}_{n,m}  =& \ 2(n-2m^2)  \, U^{\pm}_{n,m} + 2(2n+1)(n+m)(n-1+m) \,
        |x|^2 U^{\pm}_{n-2,m}   \\[1ex]
  &+ \big(  (n+m)(n-m+1)(2m-1)  \, U^{\pm}_{n,m-1} \\
   &\quad + (2n+1)(n+m)(n+m-1)(n+m-2) \, |x|^2 U^{\pm}_{n-2,m-1}   \\
   &\quad + (2m+1) \, U^{\pm}_{n,m+1} - (2n+1)(n+m)\, |x|^2 U^{\pm}_{n-2,m+1}
  \big)\,e_1   \\[1ex]
   &\mp \big(   (n+m)(n-m+1)(2m-1) \, U^{\mp}_{n,m-1} \\
   &\quad +(2n+1)(n+m)(n+m-1)(n+m-2) \, |x|^2 U^{\mp}_{n-2,m-1}   \\
   &\quad -(2m+1) \, U^{\mp}_{n,m+1} \mp (2n+1)(n+m) \, |x|^2  U^{\mp}_{n-2,m+1}   \big)\,e_2 .
  \end{align*}   
  The inframonogenic polynomials of type 2  can be expressed as
  \begin{align*}
     \Z^+_{n,0}   =&\   U^-_{n,1} e_1 - U^+_{n,1} e_2,  \\[1ex]
   (2n-1) \Z^{\pm}_{n,m}  =&\ -(2m-1)(2m+1) \, U^{\pm}_{n,m} \\[1ex]
      &+ (2n+3)(n+m)(n+m-1) \, |x|^2 U^{\pm}_{n-2,m}   \\[1ex]
 & +  \big(  (n+m)(n-m+1)(2m-1)  \, U^{\pm}_{n,m-1} \\
&\quad + \frac{(2n+3)(n+m)(n+m-1)(n+m-2)}{2} \, |x|^2 U^{\pm}_{n-2,m-1}  \\
&\quad +  (2m+1) \, U^{\pm}_{n,m+1} - \frac{(2n+3)(n+m)}{2} \, |x|^2 U^{\pm}_{n-2,m+1}  \big)\,e_1 \\[1ex]
   & \mp \big(  (n+m)(n-m+1)(2m-1) \, U^{\mp}_{n,m-1} \\
&\quad + \frac{(2n+3)(n+m)(n+m-1)(n+m-2)}{2} \, |x|^2 U^{\mp}_{n-2,m-1}  \\
&\quad -  (2m+1)  \, U^{\mp}_{n,m+1} \mp \frac{(2n+3)(n+m)}{2} \, |x|^2 U^{\mp}_{n-2,m+1} \big)\,e_2  .
\end{align*}
\end{prop}

\begin{proof}
  The expressions for $X^\pm_{n,m}$ are well known
  \cite{MoraisGurlebeck2012}. Now consider $Y^+_{n,m}$.  By
  (\ref{eq:esphecoor}),
\begin{align*}
  \overline{x}\,& X^+_{n-1,m}   + X^+_{n-1,m} \, \overline{x}
       = \cos(m\varphi)\,\big(\, 2(n+m)\cos\theta \, P_{n-1}^m(\cos\theta) \\
  &\qquad +(n+m-1)(n+m)\sin\theta\, P_{n-1}^{m-1}(\cos\theta)- \sin\theta \, 
 P_{n-1}^{m+1}(\cos\theta)\, \big)    \\  
  & + \big(\, (n+m-1)(n+m)\cos\theta \cos((m-1)\varphi)\, P_{n-1}^{m-1}(\cos\theta) \\
  &\ \ -\cos\theta \cos((m+1)\varphi)\, P_{n-1}^{m+1}(\cos\theta) \\
  &\ \ -2(n+m)\sin\theta \cos\varphi \cos(m\varphi)\, P_{n-1}^{m}(\cos\theta)\, \big)\ e_1  \\
  & - \big(\, (n+m-1)(n+m)\cos\theta \sin((m-1)\varphi)\, P_{n-1}^{m-1}(\cos\theta) \\
  &\ \ +\cos\theta \sin((m+1)\varphi)\, P_{n-1}^{m+1}(\cos\theta) \\
  &\ \ -2(n+m)\sin\theta \sin\varphi \cos(m\varphi)\, P_{n-1}^{m}(\cos\theta)\, \big)\ e_2 .
\end{align*} 
From this the scalar part of $Y^+_{n,m}$ is
\begin{align*}
  [Y^+_{n,m} ]_0 &=
  \cos(m\varphi)\, \big(\, 2(n+m)\cos\theta \ P_{n-1}^m(\cos \theta) \\
  &\ \  +(n+m-1)(n+m)\sin\theta \, P_{n-1}^{m-1}(\cos\theta) - \sin\theta \,
    P_{n-1}^{m+1}(\cos\theta) \\
  &\ \  +2(n+m)(n+m-1)\  P_{n-2}^{m}(\cos\theta)\, \big) .
\end{align*}
The well-known recurrence relations for the associated
Legendre functions
\cite{AbramowitzStegun1964,Bateman1953,Hobson1931,Sansone1959} 
are useful for manipulating expressions involving solid spherical
harmonics. We apply
\begin{align} \label{eq:rec1}
  \dfrac{1}{\sqrt{1-t^2}} P_n^m(t) =
  \dfrac{-1}{2m}(P_{n-1}^{m+1}(t)+(n+m-1)(n+m) P_{n-1}^{m-1}(t))
\end{align}
to obtain
\begin{align*}
  [Y^+_{n,m} ]_0 &=
   \cos(m\varphi) \, \big(\, 2(n+m)\cos\theta \, P_{n-1}^m(\cos \theta)-2m \,
    P_{n}^{m}(\cos\theta)\\
   &\ \ - 2\sin\theta \, P_{n-1}^{m+1}(\cos\theta) + 2(n+m)(n+m-1)\,
     P_{n-2}^{m}(\cos\theta) \,\big),
\end{align*}
by (\ref{eq:rec2})
which via the further relation
\begin{align}\label{eq:rec2}
\sqrt{1-t^2} P_n^{m+1}(t)=(n-m)t  P_{n}^{m}(t)-(n+m) P_{n-1}^{m}(t)
\end{align}
simplifies to
\begin{align*}
  [ Y^+_{n,m} ]_0 &= \cos(m\varphi) \,
                    \big(\, 2(2m+1)\cos\theta \,
  P_{n-1}^m(\cos \theta)-2m \, P_{n}^{m}(\cos\theta)\\
  &  +2(n+m-1)(n+m+1)\, P_{n-2}^{m}(\cos\theta) \big).
\end{align*}
Now apply the relation
\begin{align}\label{eq:rec3}
(n-m+1) P_{n+1}^{m}(t)=(2n+1)t P_{n}^{m}(t)-(n+m) P_{n-1}^{m}(t)
\end{align}
twice, yielding
\begin{align*}
[Y^+_{n,m} ]_0 &= \cos(m\varphi)\ (\ 2(n+m)(2n+1)\cos\theta \ P_{n-1}^m(\cos \theta) \\
&\ \ +2(m^2-n(n+1)) \ P_{n}^{m}(\cos\theta)\\
&=\dfrac{1}{2n-1} \big(2(n-2m^2)U_{n,m}^+ +2(n+m)(2n+1)(n+m-1)\ |x|^2 U_{n-2,m}^+ \big)
\end{align*}
as required. Next, the $e_1$ component of $\Y^+_{n,m}$ is
\begin{align*}
  [Y^+_{n,m}]_1 &= (n+m-1)(n+m)\cos\theta \cos((m-1)\varphi) \, 
   P_{n-1}^{m-1}(\cos\theta)\\
 & \quad \ \ -\cos\theta \cos((m+1)\varphi)\, P_{n-1}^{m+1}(\cos\theta) \\
 &\quad \ \ -2(n+m)\sin\theta \cos\varphi \cos (m\varphi)\, P_{n-1}^{m}(\cos\theta)\\
 & \quad \ \ +2(n+m)(n+m-1)(n+m-2)\ P_{n-2}^{m-1}(\cos\theta)\cos ((m-1)\varphi)\\
 & \quad \ \ -2(n+m)\ P_{n-2}^{m+1}(\cos\theta)\cos ((m+1)\varphi).
\end{align*}
Since $\cos(m\varphi)\cos\varphi=\dfrac{1}{2}(\cos((m+1)\varphi)+\cos((m-1)\varphi))$,
this is
\begin{align*}
  [Y^+_{n,m}]_1 &= (n+m) \cos((m-1)\varphi)\,
   \big( (n+m-1)\cos \theta \, P_{n-1}^{m-1}(\cos\theta)\\
 & \quad \ \ +2(n+m-1)(n+m-2)\, P_{n-2}^{m-1}(\cos\theta)-\sin\theta\,  P_{n-1}^{m}(\cos\theta)\, \big)\\
 & \quad \ \ -\cos ((m+1)\varphi)\, \big( \cos\theta \, P_{n-1}^{m+1}(\cos\theta)+2(n+m)\, P_{n-2}^{m+1}(\cos\theta)\\
 & \quad \ \ +(n+m)\sin\theta P_{n-1}^{m}(\cos\theta)\,  \big).
\end{align*}
The desired result is now obtained with \eqref{eq:rec3} together with the additional
recurrence relations  
\begin{align}\label{eq:rec4}
  \sqrt{1-t^2} P_n^{m}(t) &=
   \dfrac{1}{2n+1}((n-m+1)(n-m+2) P_{n+1}^{m-1}(t) \\
    &\quad -(n+m-1)(n+m) P_{n-1}^{m-1}(t) ),  \nonumber\\
 \sqrt{1-t^2} P_n^{m}(t)&=-\dfrac{1}{2n+1}(P_{n+1}^{m+1}(t)- P_{n-1}^{m+1}(t) ) .
\end{align}
The  $e_2$ component of $Y^+_{n,m}$ is calculated in the same way. The expression for $\Z^+_{n,0}$ is elementary and we omit the proof. We also omit the computations 
for $\Z^{\pm}_{n,m}$, which require no new ideas.
\end{proof}

In order to compute the representation of a given function in terms of the elements of an orthogonal basis, it is necessary to know the norms of the basis elements. The following result contains this information.
\begin{theo} \label{theo:products}
  The norms squared of the basic inframonogenic polynomials in $L^2(B_1(0))$
  are equal to 
\begin{align*}
  \|X^+_{n,0}\|^2_2 &= \frac{4\pi \, (n+1)}{2n+3},\\
  \|X^\pm_{n,m}\|^2_2 &= \frac{2\pi \, (n+1)}{(2n+3)}
                        \frac{(n+1+m)!}{(n+1-m)!},  \quad 1\le m\le n+1;\\
  \|Y^+_{n,0}\|^2_2 &= \frac{8\pi \, n}{(2n-3)(2n+1)(2n+3)} \Bigl((2n-3)(3n+1)  \\[-1.0ex]
  &\hspace{.4\textwidth} + (2n+1)^3n(n-1)(3n-4)\Bigr), \\
  \|Y^\pm_{n,m}\|^2_2 &= \frac{4\pi}{(2n-1)^2(2n+1)(2n+3)}
                        \frac{(n+m)!}{(n-m)!} \Bigl( (2n+1)^3(n^2-m^2)(n-1) \\
  &\hspace{.1\textwidth} + 2(n-2m^2)^2 + (n+m)(n-m+1)^2(2m-1)^2 \\
  &\hspace{.1\textwidth} + (n-m)(n+m+1)^2(2m+1)^2 \Bigr), \quad 1\le m \le n-1;   \\
  \|\Z^+_{n,0}\|^2_2 &= \frac{4\pi \, n(n+1)}{(2n+1)(2n+3)}, \\
  \|\Z^\pm_{n,m}\|^2_2 &= \frac{\pi}{(2n-1)^2(2n+1)(2n+3)} \frac{(n+m)!}{(n-m)!}  \Bigl((2m-1)^2 \left(8m+(2n+1)^2 \right) \\
  &\hspace{-.01\textwidth}  - (2m+1)^2 \left(8m-(2n+1)^2 \right) + (n-1)(2n+1)(2n+3)^2(n^2-m^2)\Bigr), \quad 1\le m \le n.
  \end{align*}
Further, $\langle X^\pm_{n,n}, Y^\pm_{n,n} \rangle=0$, $\langle X^+_{n,0}, \Z^+_{n,0} \rangle=0$, and $\langle Y^\pm_{n,m}, \Z^+_{n,0} \rangle=0$ for $0 \le m \le n-1$. The only other scalar products of the $X^\pm_{n,m}$, $Y^\pm_{n,m}$, or $\Z^\pm_{n,m}$ which are not zero are the
following:
\begin{align*}
  \langle X^\pm_{n,m}, \Z^\pm_{n,m} \rangle &=
  -\frac{2\pi}{(2n+1)(2n+3)} \frac{(n+m+1)!}{(n-m)!}, \quad 1\le m \le n,\\
\langle Y^\pm_{n,m}, \Z^\pm_{n,m} \rangle &= \frac{4\pi}{(2n-3)(2n-1)^2(2n+1)} \frac{(n+m)!}{(n-m)!} \Bigl((2n-3) \left( 4(n-1)m^2+n\right) \\
&\hspace{.12\textwidth}  + (2n+1)^2 (n^2-m^2) \left((n-1)^2 + m(m-1)\right)\Bigr), \quad 1\le m \le n-1.
\end{align*}
\end{theo}

\begin{proof}
  From
  \eqref{scalar-inner-product}--\eqref{L2norms-solid-harmonics-general},
  and Proposition \ref{prop:components}, it follows upon breaking down the integrand that
\begin{align*}
\langle X^\pm_{n,n}, \, Y^\pm_{n,n} \rangle =& \, -2n(2n+1) \int\!\!\int\!\!\int_{B_1(0)} \, (U^\pm_{n,n})^2 \, dV \\
&+ \, 4n(2n+1)^2 \int\!\!\int\!\!\int_{B_1(0)} \, U^\pm_{n,n} \, |x|^2 U^\pm_{n-2,n} \, dV \\
&+ \, 2n^2(2n+1) \int\!\!\int\!\!\int_{B_1(0)} \, \left((U^\pm_{n,n-1})^2 + (U^\mp_{n,n-1})^2 \right) dV \\
= & \, 0, \\
\langle X^+_{n,0}, \, \Z^+_{n,0} \rangle =& \, \int\!\!\int\!\!\int_{B_1(0)} \left( -U^+_{n,1} U^-_{n,1} + U^-_{n,1} U^+_{n,1} \right) dV = 0
\end{align*}
and
\begin{align*}
&\langle Y^\pm_{n,m}, \, \Z^+_{n,0} \rangle \\
&= \, \frac{(n+m)(n-m+1)(2m-1)}{2n-1} \int\!\!\int\!\!\int_{B_1(0)} \left(U^\pm_{n,m-1} \, U^-_{n,1} \pm U^\mp_{n,m-1} \, U^+_{n,1} \right) dV \\
&\quad + \, \frac{(2n+1)(n+m)(n+m-1)(n+m-2)}{2n-1} \int\!\!\int\!\!\int_{B_1(0)} \, (|x|^2 U^\pm_{n-2,m-1} \, U^-_{n,1}\\
  &\hspace{.6\textwidth} \pm |x|^2 U^\mp_{n-2,m-1} \, U^+_{n,1} ) dV \\
  &\quad + \, \frac{(2m+1)}{2n-1} \int\!\!\int\!\!\int_{B_1(0)} \, (U^\pm_{n,m+1} \, U^-_{n,1}   \mp U^\mp_{n,m+1} \, U^+_{n,1} ) dV \\
 &\quad - \, \frac{(2n+1)(n+m)}{2n-1} \int\!\!\int\!\!\int_{B_1(0)} \, (|x|^2 U^\pm_{n-2,m+1} \, U^-_{n,1} + |x|^2 U^\mp_{n-2,m+1} \, U^+_{n,1} ) dV \\
&= \, 0.
\end{align*}
Analogously, straightforward computations show that
\begin{align*}
\langle X^\pm_{n,m}, \, \Z^\pm_{n,m} \rangle =& \, -\frac{(n+m+1)(4m^2-1)}{2n-1} \, \|U^{\pm}_{n,m}\|_2^2 \\
&- \, \frac{2m+1}{2(2n-1)} \left(\|U^{\pm}_{n,m+1}\|_2^2 + \|U^{\mp}_{n,m+1}\|_2^2 \right) \\
&+ \, \frac{(n+m)^2(n+m+1)(n-m+1)(2m-1)}{2(2n-1)} \left(\|U^{\pm}_{n,m-1}\|_2^2 + \|U^{\mp}_{n,m-1}\|_2^2 \right) \\
\end{align*}
for $1\le m \le n$, and
\begin{align*}
(2n-1)^2 & \langle Y^\pm_{n,m}, \, \Z^\pm_{n,m}  \rangle =  \,  -2(n-2m^2)(4m^2-1) \, \|U^{\pm}_{n,m}\|_2^2 \\
&+ \, 2(2n+1)(2n+3)(n+m)^2(n+m-1) \, \||x|^2 U^\pm_{n-2,m}\|^2_2 \\
& + \, 2(n+m)^2(n-m+1)^2(2m-1)^2 \, \|U^\pm_{n,m-1}\|^2_2 \\
&+ \, (2n+1)(2n+3)(n+m)^2(n+m-1)^2(n+m-2)^2 \, \||x|^2 U^\pm_{n-2,m-1}\|^2_2 \\
&+ \, 2(2m+1)^2 \, \|U^\pm_{n,m+1}\|^2_2 + (2n+1)(2n+3)(n+m)^2 \, \||x|^2 U^\pm_{n-2,m+1}\|^2_2  
\end{align*}
for $1\le m \le n-1$. Using \eqref{L2norms-solid-harmonics-general} the results follow.

By Proposition \ref{prop:components}, direct computations show that
\begin{align*}
\|Y^\pm_{n,m}\|^2_2 = & \, 4(n-2m^2)^2 \|U^\pm_{n,m}\|^2_2 + 4(2n+1)^2(n+m)^2(n-1+m)^2 \||x|^2 U^\pm_{n-2,m}\|^2_2 \\
& + \, 2(n+m)^2(n-m+1)^2(2m-1)^2 \|U^\pm_{n,m-1}\|^2_2 \\
& + \, 2(2n+1)^2(n+m)^2(n+m-1)^2(n+m-2)^2 \||x|^2 U^\pm_{n-2,m-1}\|^2_2 \\
&+ \, 2(2m+1)^2\|U^\pm_{n,m+1}\|^2_2 + 2(2n+1)^2(n+m)^2 \||x|^2 U^\pm_{n-2,m+1}\|^2_2
\end{align*}
and
\begin{align*}
\|Z^\pm_{n,m}\|^2_2 = &  \, (4m^2-1)^2 \|U^\pm_{n,m}\|^2_2 + (2n+3)^2(n+m)^2(n+m-1)^2 \||x|^2 U^\pm_{n-2,m}\|^2_2 \\
&\, + 2(n+m)^2(n-m+1)^2(2m-1)^2 \|U^\pm_{n,m-1}\|^2_2 \\
&\, + \frac{1}{2} (2n+3)^2(n+m)^2(n+m-1)^2(n+m-2)^2 \||x|^2 U^\pm_{n-2,m-1}\|^2_2 \\
&\, + 2(2m+1)^2\|U^\pm_{n,m+1}\|^2_2 + \frac{1}{2}(2n+3)^2(n+m)^2 \||x|^2 U^\pm_{n-2,m+1}\|^2_2  .
\end{align*}
Using \eqref{L2norms-solid-harmonics-general} the equalities follow.
\end{proof}


Introduce the following rational constants:
\begin{align}
  \alpha_{n,m} &= \frac{n-m+1}{(n+1)(2n+1)}, \quad 1\le m \le n ,\\ 
  \beta_{n,m} &= \frac{ -(2n+3)((n-1)(m^2(2n+3) - (n(n+1)^2)-n^4)}{  2(m^2(n+4(n-1)(n+1)^2) - n(2n+1)((n-1)(n+1)^2+n^3))},\quad
                                                0\le m \le n-1
\end{align}
and for $n\ge2$ define
\begin{align} \label{New_basis}
    Z^+_{n,0} & = \Z^+_{n,0} , \nonumber \\
    Z^\pm_{n,m}   & = \Z^\pm_{n,m} +  \alpha_{n,m} X^\pm_{n,m}
      + \beta_{n,m} Y^\pm_{n,m}, \nonumber\\
     Z^\pm_{n,n} & = \Z^\pm_{n,n} \, +  \alpha_{n,m} X^\pm_{n,n};
\end{align}
and then
\begin{align}
\B_n &= \{X^\pm_{n,m}\}_m \cup
    \{Y^\pm_{n,m}\}_m \cup \{Z^\pm_{n,m}\}_m. 
\end{align} 
As always, in the definition of $X^\pm_{n,m}$ we have $0\le m\le n+1$
for the ``+'' sign and $1\le m\le n+1$ for the``+'' sign. In the
definition of $Y^\pm_{n,m}$ we have $0\le m\le n-1$ for the ``+'' sign
and $1\le m\le n-1$ for the``+'' sign. In the definition of
$Z^\pm_{n,m}$ (excluding the cases $ Z^+_{n,0}$,
$Z^\pm_{n,n}$) we have $1\le m\le n-1$ for both the ``+'' and the
``-'' signs.

Theorem \ref{theo:products} implies that $\B_n = \{B_{n,j} \colon \ 1 \leq j \leq 6n+3\}$ is a linearly
independent set, so by Theorem \ref{prop:diminfr}, $\B_n$ is a
basis of $\inf_n$ ($\B_0$ and $\B_1$ were explained previously).  An arbitrary inframonogenic function in any ball,
being real-analytic, admits a power series representation in
$x_0,x_1,x_2$ in that ball. Since this can be differentiated
term-by-term, the homogenous part of this series of degree $n$ is
inframonogenic for every $n$. Thus, we have the following:

\begin{theo}
  An inframonogenic function $f$ in any ball in $\R^3$
  admits a unique infinite series expansion
  \begin{align}    \label{eq:series}
    f = \sum_{n=0}^\infty\sum_{j=1}^{6n+3}  a_{n,j} B_{n,j}
  \end{align}
  converging uniformly on compact subsets, where $a_{n,j}\in\R$.
\end{theo}

An elementary argument based on the Monotone Convergence Theorem shows that
if $f$ is square-integrable, then \eqref{eq:series} converges in the $L^2$
sense. Thus we have

\begin{theo}
  The collection of homogeneous polynomials $\bigcup_{n=0}^\infty\B_n$
  is an orthogonal basis for the space of all square-integrable
  inframonogenic functions in any ball centered at the origin in
  $\R^3$. 
\end{theo}
 
\section{Discussion}

The reader may be interested in the somewhat tortuous route
through which the formulas for the inframonogenic polynomials were
arrived at. It began with the so-called real Zernike ball polynomials $\ZerR{n,l}{}{} \in \pol_n$, which were 
defined in $\R^3$ in \cite{Canterakis1999}. These triply-indexed polynomials are defined as products of Jacobi polynomials $P^{(0,l+1/2)}_{(n-l)/2}(2\rho^2-1)$
and solid spherical harmonics $U^{\pm}_{l,m}$. Since these products are polynomials in $\rho^2$ in the particular case $n-l=2$, it was evident that $\ZerR{n,n-2}{}{} \in \bih_{n}$. 
Therefore automatically $\partial(\ZerR{n+2,n}{}{})\partial \in \inf_{n}$.
This observation initiated the quest.  Combining with the $2n+3$
monogenic polynomials $X^\pm_{n,m}$ produces insufficient elements
to form a basis, so the quaternion-valued Zernike polynomials
$\ZerQ{n,l}{}{}$ were explored \cite{MoraisCacao2014}. These are defined by replacing
$\rho^n$ in the definition of $X^\pm_{n,m}$ with the alternative
radial function mentioned above. Again since these Jacobi polynomials are
polynomials in $\rho^2$ when $n-l=2$, it was seen that the left derivative $\partial\ZerQ{n,n-2}{}{}$ is
annihilated by $\partialbar(\cdot)\partialbar$.  Unfortunately, the resulting functions
are not $\R^3$-valued, but perhaps surprisingly, this could be
remedied simply by dropping the $e_3$ term.  At this stage, we had in
hand a list of $6n+3$ inframonogenic polynomials, but not homogeneous:
the latter type were a mixture of terms of degree $n$ and $n-2$.  As remarked earlier every homogeneous part of an inframonogenic
polynomial is inframonogenic, so we obtained elements of $\inf_n$ by dropping the terms of lower degree. But this operation loses the linear independence: these polynomials generate a
subspace of dimension $6n+2$, as the monogenic polynomial $X^+_{n,0}$ is a
linear combination of $Y^+_{n,0}$ and
the polynomial obtained from ${\hspace{.1ex}{}^{\H}\hspace{-.5ex}\mbox{Zer}_{n,n-2}^{0}}$. For this reason, we replaced it with the
function $\Z^+_{n,0}$ which is the first in the list of contragenic
functions defined in \cite{AlvarezPorter2014}. (Curiously, the remaining basic
contragenic functions are not inframonogenic.) With these adjustments
and multiplication by constants for convenience,
the definitions of $X^\pm_{n,m}$, $Y^\pm_{n,m}$, $\Z^\pm_{n,m}$ were
as given in this paper.

\appendix 
\section{Reference formulas\label{sec:formulas}}

There are several versions of ``Leibniz formulas'' for quaternionic
(or Clifford) valued functions
\cite{BDS1982,CacaoFalcaoMalonek2011,GurlebeckSprossig1997}, that is,
expansions for differential operators applied to the product of two
functions.  The noncommutativity makes the expressions much more
complicated than in the case of complex numbers.  Our purpose in this
Appendix is to prove a formula which we believe is simpler and more
useful than those previously existing in the literature.

First we list for reference some properties of the operators
$\vc\partial$ and  $\partialbar$ which are quite easy to verify,
and appear frequently in calculations relating to inframonogenic functions.

\begin{prop}\label{prop:bilateral}
\begin{align*}
  \vc\partial \vc f + \vc f \vc\partial &=
  -2(\partial_1f_1+\partial_2f_2 ), \\
  \vc\partial \vc f_0 \vc\partial  &= \vc\partial^2f_0 =
    -(\partial_1^2+\partial_2^2)f_0,\\
  \vc\partial \vc f \vc\partial  &=
  ((-\partial_1^2 +\partial_2^2)f_1 -2\partial_1\partial_2f_2)e_1 \\
                                            & \ \  + ( -2\partial_1\partial_2f_1 + (\partial_1^2 -\partial_2^2)f_2 )e_2, \\
  \partialbar f\partialbar &=
 ( (\partial_0^2+\vc\partial^2)+2\partial_0\vc\partial)f_0
   + ( \partial_0^2\vc f + \partial_0(\vc\partial\vc f+\vc f\vc\partial) + \vc\partial \vc f \vc\partial).
\end{align*}
\end{prop}
   
\begin{prop} \label{prop:bilateralderivsimple}
  \begin{align*}
    \partialbar f_0\,\partialbar &= ((\partial_0^2-\partial_1^2-\partial_2^2) + 2\partial_0\partial_1\,e_1 +  2\partial_0\partial_2\,e_2)f_0,\\
 \partialbar f_1e_1\,\partialbar &= (-2\partial_0\partial_1+(\partial_0^2-\partial_1^2+\partial_2^2)\,e_1 - 2\partial_1\partial_2\,e_2)f_1,\\   
    \partialbar f_2e_2\,\partialbar &= (-2\partial_0\partial_2 - 2\partial_1\partial_2\,e_1+(\partial_0^2+\partial_1^2-\partial_2^2)\,e_2)f_2,\\
    \partialbar \vc f\,\partialbar &= -2(\partial_0\partial_1f_1 +\partial_0\partial_2f_2) + ( (\partial_0^2-\partial_1^2+\partial_2^2)f_1 - 2\partial_1\partial_2f_2)\,e_1\\
     & \qquad\qquad +  ( (\partial_0^2+\partial_1^2-\partial_2^2)f_1 - 2\partial_1\partial_2f_1)\,e_2\\                     
    \partialbar f+f\partialbar &= 2(\partial_0f_0-\partial_1f_1-\partial_2f_2)   + 2(\partial_0f_1+\partial_1f_0)e_1 + 2(\partial_0f_2+\partial_2f_0)e_2,\\
    \partial f+f\partial &= 2(\partial_0f_0+\partial_1f_1+\partial_2f_2)
 + 2(\partial_0f_1-\partial_1f_0)e_1 + 2(\partial_0f_2-\partial_2f_0)e_2  .
  \end{align*}
\end{prop}

As a preliminary to Leibnitz rules for the two-sided operators, we
list some rules for $\vc\partial$ applied to scalar or vector valued
functions, all of which follow from short calculations:
 
\begin{prop}\label{prop:leibvec} 
  \begin{align*}
  (1L)\qquad&  \vc\partial(f_0\,g_0) && \hspace*{-10ex} = (\vc\partial f_0)g_0+f_0(\vc\partial g_0); \\
   (2L)\qquad&   \vc\partial(f_0\,\vc g) && \hspace*{-10ex} = (\vc\partial f_0)\vc g+f_0(\vc\partial\vc g); \\
   (3L)\qquad&   \vc\partial(\vc f\, g_0) && \hspace*{-10ex} = (\vc\partial \vc f) g_0 -
  ((f_1\partial_1g_0+f_2\partial_2g_0) +
   (f_2\partial_1g_0-f_1\partial_2g_0)e_3)      ;\\
   (4L)\qquad&   \vc\partial(\vc f\,\vc g)&& \hspace*{-10ex} = (\vc\partial\vc f)\vc g +
  ( f_1\vc\partial(e_1\vc g) + f_2\vc\partial(e_2 \vc g));
  \end{align*}
  \begin{align*}
    (1R)\qquad&     (f_0\,g_0)\vc\partial && \hspace*{-10ex} = (\vc\partial f_0)g_0+f_0(\vc\partial g_0); \\
    (2R)\qquad& (f_0\,\vc g) \vc\partial && \hspace*{-10ex} = \vc g(\vc\partial f_0) +f_0(\vc g\vc\partial); \\
    (3R)\qquad&  (\vc f\, g_0)\vc\partial && \hspace*{-10ex} =
  (  \vc f\vc\partial) g_0 -                                 
  ((f_1\partial_1g_0+f_2\partial_2g_0) +
  e_3(f_2\partial_1g_0+f_1\partial_2g_0));  \\
   (4R)\qquad&   (\vc f\,\vc g)\vc\partial && \hspace*{-10ex} = \vc f  (\vc g\vc\partial) +
 ( g_1(\vc f\vc e_1)\partial + g_2(\vc fe_2)\vc\partial).      
  \end{align*}
\end{prop}

One might be tempted to express formulas such as (2L) and (3L) in terms
of operators written like ``$(f_1\vc\partial e_1+f_2\partial e_2)$''
and
``$-(f_1\partial_1+f_2\partial_2) +
(f_2\partial_1+f_1\partial_2)e_3$''; however, this would violate the
established conventions $\partial_if_j=f_j\partial_i$ used in the calculation
of operators such as $f\mapsto \partial f$ and $f\mapsto f\partial$.

\begin{prop}
  [Leibnitz rules for $\partial$, $\partialbar$]\label{prop:leibnitz}
\begin{align*}
   \partial(fg) &= (\partial f)g + f (\partial g) +
       2e_3(f_1(\partial_2g)-f_2(\partial_1g)), \\
   (fg)\partial  &=  f(g\partial) +  (f\partial)g-
       2 (g_1(\partial_2f)-g_2(\partial_1 f))e_3, \\          
   \partialbar(fg) &= (\partialbar f)g + f(\partialbar g) -
       2e_3(f_1(\partial_2g)-f_2(\partial_1g)), \\   
   (fg)\partialbar &=  f(g\partialbar) + (f\partialbar)g+
       2 (g_1(\partial_2f)-g_2(\partial_1 f))e_3.  
\end{align*}
\end{prop}

\begin{proof}
  First we note that
\begin{align*}
    \vc f(\vc\partial\vc g) &=
    (-f_1(\partial_1g_1+\partial_2g_2)+f_2(\partial_1g_2-\partial_2g_1))e_1\\
    &\ \ +(f_1(\partial_1g_2-\partial_2g_1)-f_2(\partial_1g_1+\partial_2g_1))e_2,
\end{align*}
so that
\begin{align*}
  \vc f(\vc\partial\vc g) -&2e_3(f_1\partial_2g-f_2\partial_1g) = \\
 & ( (f_1\partial_1g_0+f_2\partial_2g_0)+(-f_1\partial_2g_0+f_2\partial_1g_0)e_1)
  + (f_1\vc\partial(e_1g)+f_2\vc\partial(e_2g)).
\end{align*}
Applying this after using Proposition \ref{prop:leibvec} we find that
\begin{align*}
  \partial(fg) &= (\partial_0-\vc\partial)(fg)
                 = \partial_0(fg)-\vc\partial(fg)\\
  &= \partial_0(fg) -\vc\partial(f_0g_0+f_0\vc g+\vc f g_0 +\vc f\vc g)\\
  &= \partial_0(fg) -\vc\partial(f_0g_0) -\vc\partial(f_0\vc g)
    -\vc\partial(\vc f g_0)   -\vc\partial(\vc f\vc g)\\
  &= ( (\partial_0f)g+f(\partial_0g) ) \\
    &\ \ -((\vc\partial f_0)g_0+f_0(\vc\partial g_0)) \\
    &\ \  - ((\vc\partial f_0)\vc g + f_0(\vc\partial\vc g)) )   \\
  & \ \ - (\, (\vc\partial\vc f)g_0 + ((f_1\partial_1g_0+f_2\partial_2g_0) +
   (-f_1\partial_2g_0+f_2\partial_1g_0)e_3)\,) \\
    & \ \ -(\, (\vc\partial\vc f)\vc g +
     ( f_1\vc\partial(e_1\vc g) + f_2\vc\partial(e_2 \vc g)) \,) \\
  &= (\partial_0f)g-(\vc\partial f_0)g_0 -(\vc\partial f_0)\vc g
   - (\vc\partial\vc f)g_0 -(\vc\partial\vc f)\vc g \\
  &\ \  + f(\partial_0g) -f_0(\vc\partial g_0) - f_0(\vc\partial\vc g)    - \vc f(\vc\partial\vc g)+
   2e_3(f_1(\partial_2g)-f_2(\partial_1g)),
\end{align*}
which is equal to the right hand side of the first relation.  The last relation is obtained by taking the conjugate. The other two relations are proved similarly.  
\end{proof}





\newcommand{\authorlist}[1]{\textsc{#1}}
\newcommand{\booktitle}[1]{\textit{#1}}
\newcommand{\articletitle}[1]{``#1''}
\newcommand{\journalname}[1]{\textit{#1}}
\newcommand{\volnum}[1]{\textbf{#1}}


\begin{thebibliography}{99}


\bibitem{AbramowitzStegun1964} \authorlist{M.\ Abramowitz, I.\ A.\
    Stegun, eds.}, \booktitle{Handbook of mathematical functions, with
    formulas graphs, and mathematical tables}, National Bureau of
  Standards Applied Mathematics Series \volnum{55}, Washington (1964).
  
  \bibitem{AlvarezPorter2014} \authorlist{C.\ \'Alvarez-Pe\~{n}a and R.\ M.\ Porter,} \articletitle{Contragenic functions of three variables,} \journalname{Complex Anal.\ Oper.\ Theory,} \volnum{8}:2 (2014) 409--427.
  
\bibitem{Andrews1998} \authorlist{L.\ Andrews}, \booktitle{\em Special
    Functions of Mathematics for Engineers}, SPIE Optical Engineering
  Press, Bellingham, Oxford University Press, Oxford, 1998.

\bibitem{ACL1983} \authorlist{N.\ Aronszajn, T.\ M.\ Creese, L.\ J.\ Lipkin},
  \booktitle{Polyharmonic Functions}, Clarendon Press, Oxford (1983).

\bibitem{Bateman1953} \authorlist{Bateman Manuscript Project},
  \booktitle{Higher transcendental functions, vol.\ I}, California
  Institute of Technology (1953).
  
\bibitem{BDS1982} \authorlist{F.\ Brackx, R.\ Delanghe, F.\ Sommen},
  \booktitle{ Clifford analysis}, Pitman Advanced Publishing
  Program, Boston (1982).
  
\bibitem{Cacao2004}\authorlist{I.\ Ca{\c c}{\~a}o},
  \booktitle{Constructive Approximation by Monogenic
    polynomials}. Ph.D. diss., University of Aveiro (2004).

\bibitem{CacaoGuerlebeckBock2006} \authorlist{I.\ Ca{\c c}{\~a}o, K.\
    G\"urlebeck and B.\ Bock}, \articletitle{ On Derivatives of
    Spherical Monogenics}. \journalname{Math. Comput. Modelling},
  \volnum{51:8-11}, 847--869 (2006).

\bibitem{CacaoFalcaoMalonek2011} \authorlist{I.\ Ca{\c c}{\~a}o, I.\
    Falc{\~a}o and H.\ Malonek}, \articletitle{Laguerre derivative and
    monogenic Laguerre polynomials: An operational
    approach}. \journalname{Complex Var.}, \volnum{53}, 1084--1094
  (2011).
  
  \bibitem{Canterakis1999} \authorlist{N.\ Canterakis}, \articletitle{3D Zernike Moments and Zernike Affine Invariants For 3d Image Analysis and Recognition}. \journalname{In 11th Scandinavian Conf. on Image Analysis} (1999) 85--93.
 
  
\bibitem{Dinh2021} \authorlist{D.\ C.\ Dinh}, \articletitle{On
    Structure of Inframonogenic Functions}, \journalname{Adv.\ Appl.\
    Cliffﬀord Algebras} \volnum{31} (2021) 1--9
  
\bibitem{GurlebeckSprossig1989} \authorlist{G\"urlebeck, W.\ Spr\"
    o\ss ig}, \booktitle{Quaternionic Analysis and Elliptic Boundary
    Value Problems}. Birkh\"auser, Berlin (1990).

\bibitem{GurlebeckSprossig1997} \authorlist{G\"urlebeck, W.\ Spr\"
    o\ss ig}. \booktitle{Quaternionic and Clifford Calculus for
    Physicists and Engineers}. Chichester: John Wiley \& Sons (1997).
 
 \bibitem{GuerlebeckHabethaSproessig2008}\authorlist{K.\ G\"{u}rlebeck,
    K.\ Habetha and W.\ Spr\"{o}ssig}. \booktitle{Holomorphic
    functions in the plane and $n$-dimensional space,} Birkh\"{a}user
  Verlag, Basel-Boston-Berlin, (2008).

\bibitem{GuerlebeckHabethaSproessig2016} \authorlist{K.\ G\"urlebeck, K.\
    Habetha and W.\ Spr\"o{\ss}ig}. \booktitle{Application of
    Holomorphic Functions in Two and Higher Dimensions,} Birkh\"auser
  Verlag, Basel - Boston - Berlin, (2016).
  
\bibitem{Hobson1931} \authorlist{E.\ Hobson}. \booktitle{The Theory of
    Spherical and Ellipsoidal Harmonics}. Cambridge, 1931.
 
 \bibitem{Kravchenko2003} \authorlist{V.\ Kravchenko}. \booktitle{Applied
    quaternionic analysis. Research and Exposition in
    Mathematics}. Lemgo: Heldermann Verlag, Vol. \volnum{28} (2003).
  
\bibitem{MalonekPenaSommen2010} \authorlist{H.\ R.\ Malonek, D. Peña
    Peña, F.\ Sommen}, \articletitle{Fischer decomposition by
    inframonogenic functions}. \journalname{CUBO A Mathematical
    Journal} \volnum{12:02} 189--197 (2010).
  
\bibitem{MalonekPenaSommen2011} \authorlist{H.\ R.\ Malonek, D. Peña
    Peña, F.\ Sommen}, \articletitle{A Cauchy-Kowalevski theorem for
    inframonogenic functions}. \journalname{Math.\ J.\ Okayama Univ.}
  \volnum{53} 167--172 (2011).

\bibitem{MoraisGurlebeck2012} \authorlist{J.\ Morais, K.\ Gürlebeck},
  \articletitle{Real-part estimates for solutions of the Riesz system
    in $\R^3$}. \journalname{Complex Var.\ Elliptic Equ.},
  \volnum{57}:5 (2012) 505--522.
  
\bibitem{MoraisGeorgievSproessig2014} \authorlist{J.\ Morais, S.\
    Georgiev and W.\ Spr\"o\ss ig}. \booktitle{Real Quaternionic
    Calculus Handbook}. Birkh\"auser, Basel (2014).
   
   \bibitem{MoraisCacao2014} \authorlist{J.\ Morais and I.\ Ca{\c c}{\~a}o}. \articletitle{Quaternion Zernike spherical polynomials}. \journalname{Math. Comput.,} \volnum{84} (2015) 1317--1337.
    
\bibitem{MoraisHabilitation2021} \authorlist{J.\ Morais.}
  \booktitle{A Quaternionic Version Theory related to Spheroidal
    Functions,} Habilitation thesis, TU Bergakademie Freiberg (2021).

\bibitem{MorenoAbreuBory2017} \authorlist{A.\ Moreno García, T.\
    Moreno García, R.\ Abreu Blay, J.\ Bory Reyes}, \articletitle{A
    Cauchy Integral Formula for Inframonogenic Functions in Clifford
    Analysis}.  \journalname{Adv.\ Appl.\ Clifford Algebras,}
  \volnum{27} (2017), 1147--1159.

\bibitem{Mus1968} \authorlist{J.\ Musia\l ek}, \articletitle{On
    homogeneous polyharmonic polynomials} \journalname{Comment.\
    Math.\ Prace Mat.}  \volnum{XI} (1968) 283--288.
 
\bibitem{Sansone1959} \authorlist{G.\ Sansone}, \booktitle{Orthogonal
    Functions}. Pure and Applied Mathematics, Vol. IX. Interscience
  Publishers, New York, 1959.

\bibitem{SantiestebanBlayaArciga2021} \authorlist{D.\ A.\
    Santiesteban, R.\ Abrey Blaya, M.\ P. \'Arciga Alejandre}
  \articletitle{On $(\phi,\psi)$-inframonogenic functions in Clifford
    analysis}, \journalname{Bull.\ Braz.\ Math. Soc. (N.S.)}
  published online 27 July 2021 1--17.

  
\end{thebibliography}
\end{document}